\documentclass[a4paper,12pt,epsfig]{article}
\usepackage{amsmath,amssymb,amsthm,amscd}
\usepackage{graphics}

\newtheorem{theorem}{Theorem}[section]
\newtheorem{corr}{Corollary}[section]
\newtheorem{lemma}{Lemma}[section]
\newtheorem{prop}{Proposition}[section]
\newtheorem{remark}{Remark}[section]
\newtheorem{df}{Definition}[section]
\newtheorem{ex}{Example}[section]
\newtheorem{conj}{Conjecture}[section]

\def\polhk#1{\setbox0=\hbox{#1}{\ooalign{\hidewidth\lower1.5ex\hbox{`}\hidewidth\crcr\unhbox0}}}

\def\Ind{{\rm Ind}}
\def\pt{{\rm pt}}

\def\text{{\rm }}

\def\id{{{\rm id}}}

\def\co-ind{{\operatorname{co-ind}}}

\def\indp{{\operatorname{ind_{\it p}}}}
\def\ind{{\operatorname{ind^G\,}}}
\def\inn{{\operatorname{in\,}}}
\def\tind{{\operatorname{t-ind^G\,}}}
\def\inp{{\operatorname{in_{\it p}}}}

\def\Ind{\operatorname{Ind}}

\let\to=\rightarrow

\title {Borsuk--Ulam type theorems for $G$-spaces with applications to Tucker type lemmas}

\author {Oleg R. Musin and Alexey Yu. Volovikov}

\begin{document}
\date{}
\maketitle

\begin{abstract} In this paper we consider several generalizations of the Borsuk--Ulam theorem for $G$-spaces and apply these results to Tucker type lemmas for $G$-simplicial complexes and $PL$-manifolds.
\end{abstract}

\medskip

\noindent {\bf Keywords:} Borsuk--Ulam theorem, Tucker's lemma, $G$--space,  $G$--index

\section{Introduction}

The classical Borsuk--Ulam theorem \cite{Borsuk} states that {\it for any continuous mapping \,\,\, $f:{\Bbb S}^d \to {\Bbb R}^d$ there is a point $x\in{\Bbb S}^d$
such that $f(-x)=f(x)$}. In the same paper~\cite{Borsuk} Borsuk showed that this result is equivalent to the following statement:

\medskip

\noindent{\bf Theorem A.} {\em   For any continuous mapping $f:{\Bbb B}^d \to {\Bbb R}^d$ such that $f$ is odd on the boundary
$\partial\,{\mathbb B}^d = {\mathbb S}^{d-1}$,  there exists a point $x\in{\Bbb B}^d$ such that $f(x)=0\in \mathbb R^d$.}

\medskip

In~\cite{Mus} it was shown that similar statement holds in a case when
${\mathbb S}^{d-1} = \partial M^d$ where $M^d$ is a manifold. In  \cite{MusVo} we extended this result for more general spaces. Namely we considered a space $X$ with a closed subspace $Y\subset X$ which is a free $\mathbb Z_2$-space (such a space is called bounded). In~\cite{MusVo} we discussed conditions on X and $Y$ under which for any map $f:X\to \mathbb R^d$ such that $f|_{Y}:Y\to \mathbb R^d$ is equivariant there exists a point $x\in X$ such that $f(x)=0\in \mathbb R^d$.

Note that Theorem~A follows from 
{\bf odd mapping theorem}   that states  {\it``Every continuous odd mapping $h:{\Bbb S}^d\to {\Bbb S}^d$  has odd degree''}, in particular, any odd mapping $h:{\Bbb S}^d \to {\Bbb S}^d$ can not be homotopic to zero (see \cite{AH}, \cite{Krasn55}, \cite{Krasn}, \cite[Sect. 2]{MusS}). This theorem follows also from the Conner--Floyd generalization of the Borsuk--Ulam theorem
given in~\cite{CF}.  In \cite{MusS} the odd mapping theorem has been extended for BUT--manifolds.

One more Conner--Floyd's generalization of odd mapping theorem is in their earlier paper~\cite{CF60}, where they proved that if there is an equivariant map $h:X\to Y$ of free $\mathbb Z_2$--spaces $X$ and $Y$ with the same Yang's cohomological index~\cite{Yang54} which equals $n$, then $h^*:H^n(Y;\mathbb Z_2)\to H^n(X;\mathbb Z_2)$ is a nontrivial homomorphism (of nonzero groups). In ~\cite{MusVo} we used this fact as well as other properties of indexes of the free $\mathbb Z_2$--spaces.

In this paper as a main tool in our considerations  we use indexes of $G$-spaces where $G$ is a finite group. For groups other than $\mathbb Z_2$ indexes were introduced by
Yang, Wu, Schwarz, Conner--Floyd and others.

\bigskip

In this paper we prove $G$-analogs of the above results for maps of $G$--spaces where $G$ is a finite group. 
 We use $G$-analogs of the odd mapping theorem such as the Krasnosel'skii theorem on the degree of $G$-maps of a sphere and $G$--analogs of the Conner--Floyd result on cohomological indexes of free $\mathbb Z_2$--spaces.

Let $X$ be a free $G$-space. As a main tool in our consideration we use indexes of $G$-spaces,  {\em topological index  $\tind X$}
and  {\em cohomological index} $\ind{X}$, whose definitions and properties are given in Sections~2 and 3 respectively.

In Section 4 we give first applications of cohomological index and present generalizations of Dold's results from his  highly cited paper \cite{Dold}. In particular we discuss
nonexistence of the equivariant map $X*G\to X$ for compact or finite-dimensional paracompact free $G$-space $X$ proved in \cite{Vo05} with the help of the cohomological index (alternative proof is given in \cite{Passer}).
We show  that this fact follows directly from one of Dold's result proved  in his paper \cite{Dold}. Besides we give homological versions of Dold's results proved in \cite{Dold}.

In Sections~5 and 6 we use these notions to obtain $G$--generalizations of Theorem~A. 

Assume that $G$ can act freely on ${\mathbb S}^{d-1}$. Then there is an obvious semi-free $G$-action on $\mathbb R^d$ with the unique fixed point $0\in\mathbb R^d$ (and free on $\mathbb R^d\setminus 0$). Note that the degree of any equivariant map ${\mathbb S}^{d-1}\to {\mathbb S}^{d-1}$ equals 1 modulo $|G|$ (see \cite{Krasn}). Actually, it implies that for any continues map $f:{\mathbb B}^{d}\to {\mathbb R}^{d}$ which is equivariant on the boundary the zero set $f^{-1}(0)$ is not empty. Theorem~5.5 extends this fact for a case when ${\mathbb S}^{d-1}$ is embedded into a space  $X$.

\medskip

In particular, Theorems 5.1--5.3 for $G$-spaces imply the following result for manifolds (or pseudomanifolds):

\medskip

\noindent{\bf Theorem 5.4.} {\em  Let $M^n$ be a compact connected orientable manifold (or a pseudomanifold) with the connected boundary $\partial M$, and assume that $G$ can act freely on $\partial M$. Consider a continuous mapping $f:M\to {\Bbb R}^n$ such that
$f|_{\partial M}:\partial M\to {\Bbb R}^n$ is an equivariant map, where $\mathbb R^n$ is considered as a semifree $G$-space with the unique
fixed point in the origin $0\in \mathbb R^n$. If\, $\ind \partial M=n-1$ then the zero set $Z_f=f^{-1}(0)$ is not empty.}

\medskip

In Section~6 we discuss an alternative approach for $G$-versions of the Borsuk--Ulam theorem and prove also Bourgin--Yang type theorems for $G$-spaces.







\medskip





\medskip

Last section is devoted to $G$-versions of the Tucker lemma \cite{Tucker}, which is known to be  a discrete analog of the  Borsuk--Ulam theorem.

Let $T$ be some triangulation of the $d$-dimensional ball ${\mathbb B}^d$.
We call $T$ {\it antipodally symmetric on the boundary}  if the set of simplices of $T$ contained in the boundary $\partial\,{\mathbb B}^d = {\mathbb S}^{d-1}$ of the ball ${\Bbb B}^d$ is an antipodally symmetric triangulation of  ${\mathbb S}^{d-1}$, that is if $s\subset {\mathbb S}^{d-1}$ is a simplex of $T$, then $-s$ is also a simplex of $T$.

\medskip

\noindent{\bf Theorem B.} {\em   {\bf (Tucker's lemma)} Let $T$ be a triangulation of  ${\mathbb B}^d$ that is antipodally symmetric on the boundary.
Let $$L:V(T)\to \{+1,-1,+2,-2,\ldots, +d,-d\}$$ be a labeling of the vertices of $T$ that satisfies $L(-v)=-L(v)$ for every vertex $v$ on the boundary.
Then there exists an edge in $T$ that is {\bf complementary}, i.e. its two vertices are labeled by opposite numbers.
}

\medskip

Consider also the following version of Tucker's lemma:

\medskip

\noindent{\bf Theorem C.} {\em
Let $T$ be a centrally symmetric triangulation of the sphere ${\Bbb S}^d$. Let $$L:V(T)\to \{+1,-1,+2,-2,\ldots, +d,-d\}$$ be an equivariant labeling,
i.e. $L(-v)=-L(v)$). Then there exists a complementary edge.}

\medskip

It is well known, see \cite{Mat}, that these theorems are equivalent to the Borsuk--Ulam theorem.

\medskip

Let $X$ be a finite simplicial complex, $V(X)$ denote its vertex set and $C$ be a finite set that we call the set of colors. Recall that a $C$--labeling (coloring) of $V(X)$ is a map $V(X)\to C$. In case $X$ is a finite $G$-simplicial complex and $C$ is a $G$-set we say that a $C$--labeling is {\em equivariant} if the map $V(X)\to C$ is equivariant. When $C=G\times\{1,\dots,n\}$ we call a $C$--labeling a {\em $(G,n)$--labeling}. Thus, a $(G,n)$-labeling prescribes to each vertex some pair $(g,k)$, $g\in G$, $k\in \{1,\dots,n\}$. In what follows we consider $G\times\{1,\dots,n\}$ with the following $G$-action: $h\cdot(g,k)=(hg,k)$.

 An edge of $X$ is called {\em complementary} if labels of its vertices are $(g_1,k_1)$ and $(g_2,k_2)$ with
$g_1\not=g_2$ and $k_1=k_2$.

\medskip

If $G=\mathbb Z_2\cong C_2=\{1,-1\}$ is the cyclic group of order $2$ then a $(G,n)$--labeling is a Tucker labeling. Indeed, it follows from the obvious bijection $(\pm 1,k)\leftrightarrow \pm k$ between sets $\{1,-1\}\times \{1,\ldots,n\}$ and $\{+1,-1,+2,-2,\ldots, +n,-n\}$.

\medskip

The main result of Section 7 is the following extension of Tucker's lemma:

\medskip

\noindent{\bf Theorem 7.1.} {\em Let $X$ be a simplicial complex with a free simplicial $G$-action. Then  $\tind X\geq n$\, if and only if for any equivariant $(G,n)$--labeling of the vertex set of an arbitrary equivariant triangulation of $X$ there exists a complementary edge.}

\medskip


We consider also Tucker type lemmas for bounded spaces. In particular,  Theorem 7.2 yields the following theorem for manifolds:

\medskip

\noindent{\bf Theorem 7.3.} {\em  Let $M^n$ be a connected compact orientable $PL$--manifold such that its boundary $\partial M$ is homeomorphic to the sphere ${\Bbb S}^{n-1}$. Let $T$ be a triangulation of $M$. Suppose that  there exists a free simplicial action of a finite group $G$ on $\partial T$. Then for any $(G,n)$--labeling of $V(T)$ that is an equivariant on  $\partial T$ there exists a complementary edge.}

\bigskip

In what follows we assume that all spaces in consideration are paracompact.

\section{Topological index}

Consider a group $G$ as a discrete free $G$-space. Let $J^m(G)=G*\dots *G$ be the join of $m$-copies of $G$ with the
diagonal action of $G$.

\begin{df}
Let $X$ be a free $G$-space. Topological index $\tind X$ equals minimal $n$ such that there exists an equivariant map
$X\to J^{n+1}(G)$. If no such $n$ exists, then $\tind X=\infty$.
\end{df}
Correctness of this definition follows from the fact (discussed below) that there exists no equivariant map $J^{n+1}(G)\to J^{n}(G)$.

\begin{remark} {\rm
1) We can take $E_G=G*\dots*G*\dots=J^{\infty}(G)$ as a total space of the universal $G$-bundle $E_G\to B_G$.

2) If $G=\mathbb Z_2$ then $J^{m+1}(\mathbb Z_2)$ is equivariantly homeomorphic to $\mathbb S^m$,
since $SY=Y*\mathbb Z_2$, where $SY$ is the suspension, and
\begin{center} $\mathbb S^m=S\mathbb S^{m-1}=\mathbb S^{m-1}*\mathbb Z_2=\mathbb S^{m-2}*\mathbb Z_2*\mathbb Z_2=\dots=J^{m+1}(\mathbb Z_2)$.\end{center}

3) For a cyclic group $G=\mathbb Z_q$, $q>2$, we can take in the definition of index the following sequence of test spaces:
$G$, $\mathbb S^1$, $\mathbb S^1*G$, $\mathbb S^3$, $\mathbb S^3*G$, $\mathbb S^5$, $\mathbb S^5*G, \dots$, where each odd dimensional sphere is considered with some
free action of $G=\mathbb Z_q$. For example we can use the following free action of $\mathbb Z_q$ on spheres. Let $\mathbb S^{2n-1}$ be considered as a standard unit sphere in $\mathbb C^n$. Then the generator of $G$ acts as multiplication of coordinates of $n$-tuples by $e^{2\pi i/q}$.
}
\end{remark}

The main property of the topological index:

\medskip

\begin{center}
{\it If $X\to Y$ is equivariant then $\tind X\le \tind Y$.}
\end{center}

\medskip

It is not hard to see that if $X$ is either compact, or paracompact and finite-dimensional,
then $\tind X<\infty$, in the second case $\tind X\le \dim X$. For the proof one can use nerves of $G$-invariant coverings of a $G$-space and the fact that
$J^{n+1}(G)$ is $n$-universal, i.e. any $G$-$CW$-complex of dimension not exceeding $n$ can be mapped equivariantly to $J^{n+1}(G)$. If $G$ can act freely on a sphere of some dimension $N$ then this sphere is $N$-universal.

The equality
$\tind J^{n+1}(G)=n$ provides correctness of the definition of topological index. This fact
can be proved using (co)homological index (see next section), which is a lower bound for topological index. Proofs that don't use cohomological indexes can be found in \cite{Mat}. One more or less elementary proof is sketched below.

\medskip

We need to show that there exists no equivariant map $J^{N+1}(G)\to J^N(G)$ for any $N$. Let us call continuous equivariant map of $G$-spaces also as a $G$-map. Obviously  a $G$-map $J^2(G)=G*G\to J^1(G)=G$ cannot exist, since $G*G$ is connected while $G$ is not. Thus we can assume that $N>1$.

First let $G=\mathbb Z_p$ and suppose that a $G$-map $J^{N+1}(G)\to J^N(G)$ exists for some $N>1$.
If $p=2$, then this map is just a map of spheres and the contradiction follows from the Borsuk--Ulam theorem.
If $p$ is odd then taking join with $G$ we obtain also a $G$-map $J^{N+2}(G)\to J^{N+1}(G)$. One of $N+1$ and $N+2$ is even, say equals $2m$, hence we have a $G$-map $J^{2m}(G)\to J^{2m-1}(G)$. Taking the composition of this map with the embedding $J^{2m-1}(G)\subset J^{2m}(G)$ we obtain an equivariant nullhomotopic map  $J^{2m}(G)\to J^{2m}(G)$. Now there exist equivariant maps $\mathbb S^{2m-1}\to J^{2m}(G)$ and $J^{2m}(G)\to \mathbb S^{2m-1}$. Taking composition of these three maps we obtain an equivariant nullhomotopic map  $\mathbb S^{2m-1}\to \mathbb S^{2m-1}$, a contradiction, since the degree of an equivariant map of spheres equals 1 modulo order of $G$ by Krasnosel'skii theorem (see e.g. \cite{Krasn55, Krasn}). Also at this point one can apply Dold's arguments~\cite{Dold} (based only on the notion of fixed point index) showing that an equivariant selfmap of a sphere cannot be nullhomotopic.

In general case let $G_0$ be any nontrivial subgroup of $G$. Consider $J^n(G)$ as a $G_0$-space. Then it is easy to see that $J^n(G)$  like $J^n(G_0)$ is a $(n-1)$-universal $G_0$-space. Hence there exist $G_0$-maps $J^n(G_0)\to J^n(G)$ and $J^n(G)\to J^n(G_0)$. Therefore the existence of a $G$-map $J^{N+1}(G)\to J^N(G)$ implies the existence of a $G_0$-map $J^{N+1}(G_0)\to J^N(G_0)$. To finish the proof it remains to take $G_0\cong\mathbb Z_p$ where $p$ is any prime divisor of  the order of $G$ .


\medskip

\begin{remark} {\rm For $G=\mathbb Z_2$ this index was introduced by Yang~\cite{Yang54} under the name $B$-index (Yang also introduced homological
index which is discussed below).
For finite groups topological index was introduced by M.\,Krasnosel'skii and in general case (for topological groups) by
Albert Schwarz under the name genus (more precise genus is by 1 greater than topological index). In fact Schwarz~\cite{Sv}
introduced and studied more general notion of genus of a fiber space which generalize the
notions of the Lusternik--Shnirelman category and of Krasnosel'skii genus of a covering (it is valid for a continuous surjective map).

On his web page Schwarz \cite{Sv2} writes:
"The same notion was rediscovered (under another name) 25 years later by S.\,Smale who used to estimate topological complexity
of algorithms".

Nowadays this notion is usually called sectional category.
}
\end{remark}



\section{Cohomological index}

Consider first the case of an action of the group $\mathbb Z_p$ of prime order $p$ (the case $p=2$ was considered in~\cite{MusVo}). Using Smith's sequences we can define for a free $\mathbb Z_p$-space its cohomological index
$\indp X\in \{0,1,2,\dots;\infty\}$ possessing the following properties (see \cite{Vo05} for details):

\vskip7pt

1. If there exists an equivariant map $X\to Y$ of free $\mathbb Z_p$-spaces then $\indp (X)\leq \indp (Y)$.

2. If $X=A \bigcup B$ are open invariant subspaces, then $$\indp (X)\le \indp (A)+\indp (B)+1.$$

3. Tautness: If Y is a closed invariant subspace of $X$, then there exists an open invariant neighborhood of $Y$ such that
$\indp (Y)=\indp (U)$.

4. $\indp (X)>0$ if $X$ is connected.

5. Let $X$ be either compact, or paracompact and finite dimensional. Then $\indp (X)<\infty$.

6. Assume that $X$ is connected and $H^i(X;{\mathbb Z_p})=0$ for $0<i<N$. Then $\indp (X)\ge N$.

7. Assume that $X$ is finite dimensional and $H^i(X;{\mathbb Z_p})=0$ for $i>d$. Then $\indp (X)\le d$.

8. If there exists an equivariant map $f:X\to Y$ and $\indp (X)= \indp (Y)=k<\infty$ then $0\not=f^*:H^k(Y;\mathbb Z_p)\to H^k(X;\mathbb Z_p)$.

\vskip5pt

Here \u Cech cohomology groups are used.

\medskip

For $G=\mathbb Z_2$ this index was introduced by Yang~\cite{Yang54}. In \cite{Yang57} Yang actually used this index for $G=\mathbb Z_3$ without naming it.
Conner $\&$ Floyd in~\cite{CF60} introduced for any finite group $G$ and a commutative ring with unit $L$ a cohomological index for which they used notation
$\co-ind_L(\,\cdot\,)$. At the same time A.S.\,Schwarz~\cite{Sv} introduced homological genus. It can be shown that homological genus equals \, $\co-ind_{\mathbb Z}(\,\cdot\,)+1$, and $\co-ind_{\mathbb Z_p}(\,\cdot\,)$ for $G=\mathbb Z_p$ coincides with $\indp(\,\cdot\,)$.

\medskip

In what follows the property~8 will serve as our main tool.

For example, from properties~1 and 8 we immediately obtain that $\indp(\,\cdot\,)$ is stable, i.e.
$\indp\, X*\mathbb Z_p=\indp X+1$ (see \cite[Corollary~3.1]{Vo05}), so if $\indp X$ is finite then there exists no equivariant map $X*\mathbb Z_p\to X$.

\bigskip

Now we recall the definition of $\indp(\,\cdot\,)$. Denote by $\pi:X\to X/\mathbb Z_p$ the  projection. Then
there are two Smith sequences (see e.g. \cite{Bredon}):
$$
\begin{CD}
\dots\to H^k_{\rho}(X) @>>> H^k(X) @>\pi^{!}>> H^k(X/\mathbb Z_p) @>\delta_1>> H^{k+1}_{\rho}(X)\to\dots
\end{CD}
$$
and
$$
\begin{CD}
\dots\to H^k(X/\mathbb Z_p) @>\pi^*>> H^k(X) @>>> H^k_{\rho}(X) @>\delta_2>> H^{k+1}(X/\mathbb Z_p)\to\dots
\end{CD}
$$
Here coefficients $\mathbb Z_p$ are omitted, groups $H^*_{\rho}(X)$ are called special Smith cohomology groups, and $\pi^{!}$ is called the transfer (see \cite{Bredon}).

Let us define $s_{2d}:H^0(X/\mathbb Z_p)\to H^{2d}(X/\mathbb Z_p)$ and $s_{2d+1}:H^0(X/\mathbb Z_p)\to H^{2d+1}_{\rho}(X)$ as
$s_{2d+1}=\delta_1 s_{2d}$ and $s_{2d+2}=\delta_2 s_{2d+1}$ where $s_0=\id$, and put $u_n(X)=s_n(1)$, $1\in H^0(X/\mathbb Z_p)$.
Then $\indp X$ equals maximal $n$ such that $u_n(X)\not=0$.

The following proposition is a partial converse to  Property~8 (see also \cite[Proposition~3.3]{Vo00}).

\begin{prop}
Let $X$ and $Y$ be free $\mathbb Z_p$-spaces and $f: X\to Y$  an equivariant map.
Assume that $k$ is odd and

a) $\indp (Y) = k$,

b) $\dim X=k$,

c) $H^k(X;\mathbb Z_p)=H^k(Y;\mathbb Z_p)=\mathbb Z_p$,

d)  $f^*:H^k(Y;\mathbb Z_p)\to H^k(X;\mathbb Z_p)$ is an isomorphism.

Then \, $\indp(X) = k$.
\end{prop}
\begin{proof} Put $k=2n+1$.

An equivariant map $f:X\to Y$ between free $\mathbb Z_p$-spaces induces
a map of factor spaces $X/\mathbb Z_p\to Y/\mathbb Z_p$ and we have two commutative diagrams (for $p$ odd) since Smith's sequences are functorial.
Consider one of these diagrams:
$$
\begin{CD}
H^k(X/\mathbb Z_p) @>\pi^*>> H^k(X) @>>> H^k_{\rho}(X) @>\delta_2>> H^{k+1}(X/\mathbb Z_p)\\
@AAA @AAf^*A @AAA @AAA \\
H^k(Y/\mathbb Z_p) @>>> H^k(Y) @>>> H^k_{\rho}(Y) @>>\delta_2> H^{k+1}(Y/\mathbb Z_p) \end{CD}
$$
Since $u_k(Y)\not=0$ and $\delta_2 u_k(Y)=0$, there is a nontrivial element $\alpha\in H^k(Y)=\mathbb Z_p$ which is mapped onto $u_k(Y)$.
Now $u_k(Y)$ is mapped to $u_k(X)$ and from assumption d) it follows that $0\not=f^*\alpha \in H^k(X)=\mathbb Z_p$
is mapped onto $u_k(X)$. Now we argue by contradiction. If $u_k(X)=0$ then $H^k(X) \to H^k_{\rho}(X)$ is trivial. Since $H^{k+1}(X/\mathbb Z_p)=0$,
we obtain $H^k_{\rho}(X)=0$. We have also $H^{k+1}_{\rho}(X)=0$, since $\dim X=k$. From Smith's sequence
$$
\begin{CD}
H^k_{\rho}(X) @>>> H^k(X) @>\pi^{!}>> H^k(X/\mathbb Z_p) @>\delta_1>> H^{k+1}_{\rho}(X)
\end{CD}
$$
we see that $\pi^{!}$ is an isomorphism and $H^k(X/\mathbb Z_p)=\mathbb Z_p$. From the first row of the above diagram it follows that
$\pi^*:H^k(X/\mathbb Z_p) \to H^k(X)$ is also an isomorphism, so $\pi^{!}\circ\pi^*$ is an isomorphism, but this contradicts with the fact that $\pi^{!}\circ\pi^*$ is the multiplication by $p$, i.e. zero homomorphism.
\end{proof}

Note that if $X$ is a free $\mathbb Z_p$-space where $p$ is an odd prime and $\dim X=2n+1$, then there exists an equivariant map $f: X\to \mathbb S^{2n+1}$.

\begin{corr}
Let $X$ be a free $\mathbb Z_p$-space where $p$ is an odd prime. Assume that $\dim X=2n+1$ and $H^{2n+1}(X;\mathbb Z_p)=\mathbb Z_p$, and denote by $f: X\to \mathbb S^{2n+1}$ an equivariant map. Then
$\indp(X) = 2n+1$ if and only if $f^*:H^{2n+1}(\mathbb S^{2n+1};\mathbb Z_p)\to H^{2n+1}(X;\mathbb Z_p)$ is an isomorphism.
\end{corr}

\bigskip

In what follows we will use cohomological index with integer coefficients. This index is defined via homological genus introduced by Albert Schwarz in ~\cite{Sv}.

\begin{df}
Let $X$ be a free $G$-space. We define $\ind X$, the integer cohomological index of $X$, as  its Schwarz's homological genus minus 1.
\end{df}

\begin{remark} {\rm 1) Using notation of Conner and Floyd~\cite{CF60} we have $\ind (\,\cdot\,)=\co-ind_{\mathbb Z}(\,\cdot\,)$.

2) $\ind (\,\cdot\,)$ is the largest cohomological index. In particular for $G=\mathbb Z_p$ we have $\indp(\,\cdot\,)\le \ind (\,\cdot\,)$. Also for any $G$
we have $\ind (\,\cdot\,)\le \tind (\,\cdot\,)$.

}
\end{remark}

This cohomological index possesses similar properties:

\vskip7pt

1. If there exists an equivariant map $X\to Y$ then $\ind (X)\leq \ind (Y)$.

2. If $X=A \bigcup B$ are open invariant subspaces, then $$\ind (X)\le \ind (A)+\ind (B)+1.$$

3. Tautness: If Y is a closed invariant subspace of $X$, then there exists an open invariant neighborhood of $Y$ such that
$\ind (Y)=\ind (U)$.

4. $\ind (X)>0$ if $X$ is connected.

5. If $X$ is either compact, or paracompact and finite dimensional then $\ind (X)<\infty$.

6. Assume that $X$ is connected and $H^i(X;{\mathbb Z})=0$ for $0<i<N$. Then $\ind (X)\ge N$.

7. Assume that $X$ is finite dimensional and $H^i(X;{\mathbb Z})=0$ for $i>d$. Then $\ind (X)\le d$.

8. If there exists an equivariant map $f:X\to Y$ and $\ind (X)= \ind (Y)=k<\infty$ then $0\not=f^*:H^k(Y;\mathbb Z)\to H^k(X;\mathbb Z)$.


\section{Dold theorems and generalizations}

Note that from properties~1 and 8 we immediately obtain that $\indp(\,\cdot\,)$ is stable, i.e.
$\indp\, X*\mathbb Z_p=\indp X+1$ (see \cite[Corollary~3.1]{Vo05}), so if $\indp X$ is finite then there exists no equivariant map $X*\mathbb Z_p\to X$.
As a direct consequence we have the following assertion:
\begin{prop}\label{partsol} Let $H$ be any topological group which has a nontrivial finite subgroup and
$X$ be either compact or paracompact and finite dimensional  space with a free action of $H$.
Then there exists no equivariant map $X*H\to X$.
\end{prop}
 An independent, alternative proof of this result is given by Passer in \cite{Passer}. One of his arguments is used below in a more simple proof of this proposition. Also we show below that proposition~\ref{partsol} follows directly from the paper of Dold~\cite{Dold}.

Proposition~\ref{partsol} gives the partial solution to the following conjecture of Baum, D\polhk{a}browski and Hajac:
\begin{conj}[\cite{BDH}, Conjecture 2.2] {\rm Let $X$ be a compact Hausdorff
space with a continuous free action of a nontrivial compact Hausdorff group $G$. Then,
for the diagonal action of $G$ on the join $X*G$, there does not exist an equivariant continuous map
$f : X * G \to X$.
}
\end{conj}

In \cite{ChPass}, Chirvasitu and Passer proposed a possible approach to the open part of Conjecture~4.1 (and its analogue for compact group actions on $C^*$-algebras) using the ideas of \cite{Passer} and \cite{Dold}. The case of certain compact \textit{quantum} group actions on $C^*$-algebras was considered by D\polhk{a}browski, Hajac, and Neshveyev in \cite{DHN}.

Let us deduce proposition~\ref{partsol} directly from Dold's~\cite{Dold} result and give one more simple proof.


\medskip

Dold in the proof of his last theorem in paper~\cite{Dold} showed for finite group $G\not=\{1\}$ that

\medskip

{\it If there exists an equivariant nullhomotopic map $X\to X$ of a free $G$-space $X$ to itself then for every free $G$-space $Y$ such that $Y/G$ is a finite cell complex there exists an equivariant map $Y\to X$.}

\medskip

\begin{proof}[Proposition~\ref{partsol} is a consequence of this Dold's result]
In particular we can take $Y=J^N(G)$ with any $N$ and obtain an equivariant map $J^N(G)\to X$. Thus if there exists an equivariant nullhomotopic selfmap $X\to X$ then $\tind X=\infty$. Also it follows that an equivariant map $Y\to X$ exists for any free $G$-space $Y$ such that $\tind Y<\infty$ because we can take the composition of maps $Y\to J^n(G)$ and $J^n(G)\to X$ where $\tind Y=n-1$.


Now if we assume equivariant map $F:X*G\to X$ then its composition with the natural embedding $X\subset X*G$ gives us the equivariant nullhomotopic map $X\to X$,  so $\tind X=\infty$.
Thus in case $\tind X<\infty$ a $G$-map $X*G\to X$ does not exist. In particular no such a $G$-map exists for compact $X$  (with no restriction on dimension)
and for finite dimensional paracompact $X$ since in these cases $\tind X<\infty$. 
\end{proof}

\medskip

Note that for existence of a $G$-map $X*G\to X$ some restrictions on $X$ are needed since there is an obvious equivariant homeomorphism $X*G\approx X$ for $X=J^{\infty}(G)$.
Note also that if a  $G$-map $X*G\to X$  exists then we can prove the equality $\tind X=\infty$ by the following simple argument used in \cite{Passer}\footnote{The authors thanks Benjamin Passer for his useful comments on the first version of this paper. The discussion with him led to our better understanding of the problem.}.
Taking the join with $G$ we obtain an equivariant map $X*G*G\to X*G$, and hence a map $X*G*G\to X$.
Iterating this procedure we obtain for any $n$ an equivariant map $X*J^n(G)\to X$ (this argument was used in \cite{Passer}). Since $J^n(G)$ is a $G$-subspace of
$X*J^n(G)$ we obtain an equivariant map $J^n(G)\to X$ for any $n$, and therefore $\tind X=\infty$.

\medskip

The following assertion clarifies the situation.

\begin{lemma}\label{dold} Let $X$ and $Y$ be free $G$-spaces. There exists an equivariant map $F:X*G\to Y$ if and only if there exists a nullhomotopic equivariant map $f:X\to Y$.
\end{lemma}
\begin{proof} Given $F:X*G\to Y$ we can define the equivariant nullhomotopic map $f:X\to Y$ as a composition of the natural embedding $X\subset X*G$ with $F$, i.e. $f=F|_X$.

Now let $f:X\to Y$ be an equivariant nullhomotopic map. Elements of $X*G$ are written as $[x,t,h]$, where $x\in X$, $t\in [0,1]$, $h\in G$, and
 $[x,0,h]=[x,0,e]$ and $[x,1,h]=[x',1,h]$ for any $x,\,x'\in X$ and $h\in G$. Then $G$ acts on $X*G$ as $g[x,t,h]=[gx,t,gh]$, $g\in G$, and there is an equivariant inclusion of $X$ into $X*G$ given as $x\mapsto [x,0,e]$. Denote by $f_t$ a homotopy between the equivariant map $f=f_0$ and a constant map $f_1$ such that $f_1(X)=\{y\}$, where $y\in X$ is some point.
Define $F:X*G\to Y$ by the formula $F([x,t,h]):=hf_t(h^{-1}x)$. 

We have $F([x,0,h])=hf_0(h^{-1}x)=hf(h^{-1}x)=f(x)$  and $F([x,1,h])=hf_1(h^{-1}x)=hy$, so $F$ is correctly defined.
The following calculation
\begin{center}
$
F(g[x,t,h])=F([gx,t,gh])=ghf_t((gh)^{-1}gx)=ghf_t(h^{-1}x)=gF([x,t,h])
$
\end{center}
shows that $F$ is equivariant.
\end{proof}


\begin{remark} {\rm   Gottlieb~\cite{Gott} proved that the order of $G$ divides the Lefschetz number of an equivariant selfmap of a finitely dominated manifold with a free $G$-action. As a corollary~\cite[Corollary 4]{Gott} he obtained that no equivariant nullhomotopic selfmap of a finitely dominated manifold with a free $G$-action exists when $G\not=\{1\}$. Dold~\cite{Dold} deduced his more general result from the partial case that there exists no equivariant nullhomotopic selfmap of a sphere with a free $G$-action, and therefore if there  exists an equivariant map of spheres $\mathbb S^n\to \mathbb S^N$ with free $G$-actions then $n\le N$. Dold's argument (calculation of fixed point indices of a map of factor spaces) for a selfmap of a sphere is just the same as Gottlieb's for a selfmap of a compact manifold. As was mentioned above this result for spheres follows also from earlier stronger theorem of Krasnosel'skii~\cite{Krasn55, Krasn} who proved that the degree of an equivariant map of a sphere to itself is 1 modulo the order of $G$.
}
\end{remark}

Dold proved (see Remark  and Theorem on page 68 in~\cite{Dold}) the following:

\medskip

{\it If a map $f:X\to Y$ commutes with some free actions of a notrivial finite group $G$ on $X$ and $Y$ then
$$
\dim Y\ge 1+ {\rm Connectivity}(X).
$$

If $\dim Y= 1+ {\rm Connectivity}(X)<\infty$
then $f$ is not nullhomotopic (assuming $Y$ paracompact).

If $X$ is a finite-dimensional paracompact space and $f:X\to X$
is a nullhomotopic map then $f$ does not commute with any free $G$-action on $X$ for any finite group $G\not=\{1\}$.
}


\medskip

Here are homological versions of these results.

\begin{prop} Let $H$ be a subgroup of $G$ of prime order
$p$ and denote by $n=\indp X$ the cohomological index of a space $X$ with respect to $H$. If $0<n<\infty$, then $H^n(X;\mathbb Z_p)\not=0$ and the induced endomorphism $f^*:H^n(X;\mathbb Z_p)\to H^n(X;\mathbb Z_p)$ is nontrivial.
\end{prop}


\begin{prop}\label{acyclic}
Let $X$ and $Y$ be free $G$-spaces and $p$ is a prime divisor of the order of $G$. Assume that $\widetilde H^i(X;\mathbb Z_p)=0$ for $i\le n$
and that $f:X\to Y$ is an equivariant map.
Then $\tind Y\ge n+1$, in particular $\dim Y\ge n+1$.
If\, $\tind Y=n+1$, then $f^*:H^{n+1}(Y;\mathbb Z_p)\to H^{n+1}(X;\mathbb Z_p)$ is a nontrivial homomorphism.
\end{prop}
\begin{proof}
The problem reduces to the case $G=\mathbb Z_p$.
Then the first assertion follows from properties 1 and 6 of the index $\indp(\,\cdot\,)$ and the fact that $\indp Y\le \tind Y\le \dim Y$, where $Y$ is a free finite-dimensional $\mathbb Z_p$-space. In particular it follows that $\indp J^{n+1}(G)=\tind J^{n+1}(G)=n$.

For the proof of the second assertion note that there exists an equivariant map $h:Y\to J^{n+2}(G)$, and from property~8 it follows that the composition $h\circ f$ induces a nontrivial homomorphism of $(n+1)$-dimensional cohomology groups (with $\mathbb Z_p$-coefficients). Therefore $f^*\not=0$ in dimension $n+1$.

 Actually for the proof the first assertion it is easier to use more simple index $\inp(\,\cdot\,)$ which equals weak homological genus (introduced in \cite{Sv}) minus 1.

To define $\inp(X)$ for a paracompact free $\mathbb Z_p$-space $X$ consider an equivariant map $X\to J^{\infty}(\mathbb Z_p)=E_{\mathbb Z_p}$ and the map of factor-spaces $\mu:X/\mathbb Z_p\to B_{\mathbb Z_p}$. Recall that $H^i(B_{\mathbb Z_p};\mathbb Z_p)=\mathbb Z_p$. Say that $\inp(X)\ge n$ if $\mu^*\not= 0$ in
dimension $n$. It is easy to see that this assumption is equivalent to the assumption that $\mu^*:H^i(B_{\mathbb Z_p};\mathbb Z_p)\to H^i(X/\mathbb Z_p;\mathbb Z_p)$ is a monomorphism for $i\le n$. To prove proposition~\ref{acyclic} we need only to show that index $\inp(\,\cdot\,)$ satisfies properties 1, 5 and 6. The most complicated property~6 follows easily from the consideration of the spectral sequence of a covering $X\to X/\mathbb Z_p$ (from the spectral sequence of a bundle
$X\times_{\mathbb Z_p} E_{\mathbb Z_p}\to X/\mathbb Z_p$ with fiber $X$).
\end{proof}

\begin{remark} {\rm 1) Results like the first statement of proposition~\ref{acyclic} (generalizations can be obtained using \cite[Theorem~17]{Sv}) belong to A.S.\,Schwarz, since they follow trivially from \cite[Theorem~17]{Sv} and its corollaries 1 and 2 and properties of homological and weak homological genus introduced in \cite{Sv}.


2)  $\inp(\,\cdot\,)$ possesses all other properties except 2 and 8, and it can be shown that $\inp(\,\cdot\,)\le \indp(\,\cdot\,)$ with the equality for $p=2$, see \cite{Vo05}.

3) The definition of $\inp(\,\cdot\,)$ and its property 6 was rediscovered many times, see e.g. \cite{Lusk}, \cite{Vo79, Vo80}. In \cite{Liulevicius} Liulevicius actually used this index without naming it.

4) It is easy to deduce from proposition~\ref{acyclic} that $\inp\, \mathbb S^n=\indp\, \mathbb S^n=\tind \mathbb S^n=n$ and
$\inp\, J^n(G)=\indp\ J^n(G)=\tind J^n(G)=n-1$ where $p$ is a prime divisor of the order of $G$ and cohomological indicies are taken in respect with any subgroup of $G$ of prime order $p$.
}
\end{remark}




\section{Borsuk--Ulam type theorems for bounded spaces}

\begin{df}
We say that $h:X_0\to X$ is $n$-cohomological trivial ($n$-c.t. map) over $R$ if $h^*:H^n(X;R)\to H^n(X_0;R)$ is the trivial homomorphism of cohomology groups with coefficients in $R$ in dimension $n$.
In the case when $h$ is an embedding we call $X_0$ an $n$-c.t.-subspace of $X$ over $R$.
\end{df}

\begin{ex}
Let $X$ be a compact connected $(n+1)$-dimensional manifold with the connected boundary $\partial X=X_0$. Then $X_0$ is an
$n$-c.t.-subspace of $X$ over $\mathbb Z_2$, and if moreover $X$ is orientable then $X_0$  is an
$n$-c.t.-subspace of $X$ over $R$ for any $R$.
\end{ex}

Let a space $X_0$ be a subspace of $X$. Denote by $i:X_0\to X$ the inclusion. Suppose $X_0$ admits
a free action of a finite group $G$. (Actually, we  do not assume that $X$ is a $G$-space.)

These assumptions on $X$ and $X_0$ will be used in what follows.

\begin{theorem}\label{ZYP} Let $Y$ be a $G$-space, $Y_0$ its invariant closed subspace such that the action on $Y\setminus Y_0$ is free, and $f:X\to Y$ a
continuous map. Assume that

1) $n=\ind X_0 = \ind (Y\setminus Y_0)$,

2) 
$X_0$  is an
$n$-c.t.-subspace of $X$ over $\mathbb Z$,

3) $f|_{X_0}:X_0\to Y$ is equivariant,

then $f^{-1}(Y_0)\not=\emptyset$.
\end{theorem}
\begin{proof}

We argue by contradiction. If $f^{-1}(Y_0)=\emptyset$ then $f$ maps $X$ into $Y\setminus Y_0$ and
$f|_{X_0}:X_0\to Y\setminus Y_0$ is equivariant. Since $f|_{X_0}=f\circ i$ and $i^*$ is trivial in dimension $n$, we obtain that
$(f|_{X_0})^*:H^n(Y\setminus Y_0; \mathbb Z)\to H^n(X_0; \mathbb Z)$ is trivial, a contradiction with property~8 of index.
\end{proof}

Note that if $\ind(Y\setminus Y_0)<n$ then by property~1 of index there exists no equivariant map from $X$ to $Y\setminus Y_0$, hence
$(f|_X)^{-1}(Y_0)\not=\emptyset$ (in this case we don't need the assumption 2).

\medskip

The theorem follows also from the following result.
\begin{theorem} Let $X_0$ be a free $G$-space, $i:X_0\subset X$. Let $K$ be a free $G$-space and $f:X\to K$ is a map equivariant on $X_0$.
Assume that $\ind  X_0=d$ and that $X_0$ is $d$-$\mathbb Z$-c.t.-subspace of $X$.
Then $\ind  K \ge d+1$.

If in addition $K$ is a connected closed orientable topological $(d+1)$-dimensional manifold or a pseudomanifold then for any $y\in K$ at least one of the sets
$f^{-1}(gy)$ for some $g\in G$ depending on $y$ is nonempty.
\end{theorem}
\begin{proof} The map $f\circ i:X\to K$ is equivariant, so $\ind  K \ge \ind X= d$. Since $(f\circ i)^*=i^*\circ f^*=0$ in
dimension $d$, it follows from property~8 of index that $\ind K \not= d$. Therefore $\ind K \ge d+1$.

When $K$ is a manifold we argue by contradiction. Let $y\in K$ be a point such that $f^{-1}(Gy)=\emptyset$ where $Gy$ is the orbit of the point $y$.
Then $f$ maps $X$ to
$K\setminus Gy$ and $f\circ i:X_0\to K\setminus Gy$ is equivariant. Applying the first statement we obtain that
$\ind (K\setminus Gy) \ge d+1$.
On the other hand $K\setminus Gy$ is an open manifold, hence $H^j(K\setminus Gy;\mathbb Z)=0$ for $j\ge d+1$, and from property~7
of index we obtain $\ind (K\setminus Gy) < d+1$. (Also $H^{d+1}(K\setminus Gy;\mathbb Z)=0$ contradicts with $\ind (K\setminus Gy) = d+1$ by
property~8.)
\end{proof}

\begin{df}  Let $Y$ be a $G$-space. A point $y\in Y$ is {\it a fixed point} of the action if $gy=y\,\,\,\, \forall g\in G$.
Denote the set of fixed points by $Y^G$. We say that the action of $G$ on $Y$ is {\it semifree} if
$Y\setminus Y^G\not=\emptyset$ and $Y^G\not=\emptyset$ and $G$ acts freely on $Y\setminus Y^G$.
\end{df}

Assume that $Y$ is a semifree $G$-space and $f:X\to Y$ a
continuous map. In this case directly from theorem~\ref{ZYP} we obtain:

\begin{theorem} Let $Y$ is a semifree $G$-space, $f:X\to Y$ a
continuous map.
Assume that

1) $n=\ind X_0=\ind (Y\setminus Y^G)$,

2) $X_0$ is $n$-c.t.-subspace of $X$ over $\mathbb Z$.

3) $f|_{X_0}:X\to Y$ is equivariant.

\noindent
Then $f^{-1}(Y^G)\not=\emptyset$.
\end{theorem}

We can apply this result in the case when $X=M$ is a manifold and $X_0=\partial M$ is its boundary.

\begin{theorem} Let $M^n$ be a compact connected orientable manifold (or a pseudomanifold) with the connected boundary $\partial M$,
and assume that $G$ can act freely on $\partial M$. Consider a continuous mapping $f:M\to {\Bbb R}^n$ such that
$f|_{\partial M}:\partial M\to {\Bbb R}^n$ is an equivariant map, where $\mathbb R^n$ is considered as a semifree $G$-space with the unique
fixed point $0\in \mathbb R^n$, the origin. If\, $\ind \partial M=n-1$ then the zero set $Z_f=f^{-1}(0)$ is not empty.
\end{theorem}

Here we consider any semifree action of $G$ on $\mathbb R^n$ with the unique fixed point $0\in \mathbb R^n$, the origin.
Such an action exists since we assume that $G$ can act freely on ${\Bbb S}^{n-1}$.
For example we can take the action which is obtained by linearity from the $G$-action on ${\Bbb S}^{n-1}$.

As a partial case of the previous assertion we obtain:

\begin{corr} Let $M^n$ be a compact connected orientable manifold (or a pseudomanifold) with the boundary $\partial M$ which is homeomorphic to the sphere
${\Bbb S}^{n-1}$, and assume that $G$ can act freely on $\partial M\approx {\Bbb S}^{n-1}$. Consider a continuous mapping $f:M\to {\Bbb R}^n$ such that
$f|_{\partial M}:\partial M\to {\Bbb R}^n$ is an equivariant map, where $\mathbb R^n$ is considered as a semifree $G$-space with the unique
fixed point $0\in \mathbb R^n$, the origin. Then the zero set $Z_f=f^{-1}(0)$ is not empty.
\end{corr}

This follows also from

\begin{prop} If there is an embedding $i:{\mathbb S}^{d-1}\to X$ such that
\begin{center}
${\rm Im}\, i^*\cap \{k\in \mathbb Z\,\,|\,\,k\equiv 1\mod |G|\}=\emptyset$,
\end{center}
where
$i^*:H^{d-1}(X;\mathbb Z)\to H^{d-1}({\mathbb S}^{d-1};\mathbb Z)$,
and $f:X\to \mathbb R^d$ a continuous map such that $f|_{{\mathbb S}^{d-1}}:{\mathbb S}^{d-1}\to {\mathbb R}^{d}$ is equivariant, then
$0\in f(X)$.
\end{prop}

Actually a more general assertion holds

\begin{theorem}  Assume that there is a map $j:{\mathbb S}^{d-1}\to X$ such that
\begin{center}
${\rm Im}\, j^*\cap \{k\in \mathbb Z\,|\,k\equiv 1\mod |G|\}=\emptyset$,
\end{center}
where
$j^*:H^{d-1}(X;\mathbb Z)\to H^{d-1}({\mathbb S}^{d-1};\mathbb Z)$ is induced by $j$,
and let $f:X\to \mathbb R^d$ be a continuous map such that $f\circ j:{\mathbb S}^{d-1}\to {\mathbb R}^{d}$ is equivariant. Then
$0\in f(X)$.
\end{theorem}
\begin{proof} We argue by a contradiction. If $0\notin f(X)$, then $f\circ j:{\mathbb S}^{d-1}\to \mathbb R^d\setminus 0$ is an equivariant map, hence its degree equals 1 modulo $|G|$ (see e.g. \cite{Krasn}), but this contradicts with the assumption ${\rm Im}\, j^*\cap \{k\in \mathbb Z\,|\,k\equiv 1\mod |G|\}=\emptyset$.
\end{proof}


\section{Bourgin--Yang type theorems}

\begin{df}
Let $X$ be a space and $X_0$ its subspace which is a $G$-space.
A {\it camomile} $C$ is a $G$-space for which there is an embedding $X\subset C$ such that
$C=GX$, induced embedding $X_0\subset C$ is equivariant, the action of $G$ on $C\setminus X_0$ is free, and
$C\setminus X_0=\bigcup\limits_{g\in G} g(X\setminus X_0)$.
\end{df}

\begin{ex} {\rm
If $X$ is a cone over $X_0$, i.e. $X=X_0*\pt$, then $C=X_0*G$.}
\end{ex}

\vskip3pt

Let $Y$ be a $G$-space and $Y_0$ its invariant subspace such that the $G$-action on $Y\setminus Y_0$ is free.
From the definition of camomile we easily obtain the following assertion.

\begin{theorem} There exists $f:X\to Y$ equivariant on $X_0$ and such that
$f^{-1}(Y_0)=\emptyset$ if and only if
there exists an equivariant map $C\to Y\setminus Y_0$ where $C$ is the camomile associated
with the embedding $X_0\subset X$ of the $G$-space $X_0$ into $X$.
\end{theorem}


\begin{theorem} Let $X_0$ be an $n$-c.t.-subspace of $X$ over $\mathbb Z$ such that $\ind X_0=n$.
Then $\ind C=n+1$.
\end{theorem}
\begin{proof}
Since the inclusion $X_0\subset C$ is equivariant, we have from property~1 that $\ind C\ge n$ and from property~8 obtain that
$\ind C\ge n+1$. By property~3 there exists an invariant neighborhood of $X_0$ in $C$ of index $n$. A complement of this neighborhood is a $G$-space
that can be mapped equivariantly to $G$, so its index equals zero. Hence from property~2 we obtain that $\ind C\le n+1$.
\end{proof}

Now we show how to construct a camomile in the case when
$X$ is a finite-dimensional compact space and $X_0$ its closed subspace (so $X_0$ is a compactum also).

By Mostow theorem~\cite{Mostow} we can equivariantly embed $X_0$ into finite-dimensional Euclidean $G$-space $V$. By Tietze lemma
we can extend this embedding to the map $\varphi:X\to V$. If $\dim X=k$ then using N\"obeling--Pontryagin theorem (see e.g. \cite{AF}) we can embed
$X$ into the unit sphere $S^{2k+1}\subset \mathbb R^{2k+2}$. Denote
this embedding by $\psi:X\to \mathbb R^{2k+2}$.
Define a real-valued function $h:X\to \mathbb R$ as $h(x)=\rho(x,X_0)$, the distance between a point $x$ and $X_0$. This function takes zero values on $X_0$ and
is positive on $X\setminus X_0$. Define $\eta:X\to \mathbb R^{2k+2}$ as $\eta(x)=h(x)\psi(x)$. Then $\zeta:X\to V\oplus \mathbb R^{2k+2}$,
$\zeta(x)=(\varphi(x),\eta(x))$, is an embedding. We will consider $V\oplus \mathbb R^{2k+2}$ as a Euclidean $G$-space ($G$ acts trivially on $\mathbb R^{2k+2}$).
Then $\zeta$ is the embedding which is equivariant on $X_0$. Finally put $W=V\oplus \mathbb R^{2k+2}\oplus \mathbb R[G]$ where $R[G]$ is the group ring considered as
Euclidean space of dimension $|G|$, the order of $G$. The group acts on $R[G]$ by left multiplication and it is convenient to denote basis vectors as
elements of the group $G$, so $\mathbb R[G]=\oplus_{g\in G}\mathbb R\cdot g$. Now we define an embedding
$\mu:X\to V\oplus \mathbb R^{2k+2}\oplus \mathbb R\cdot e\subset W$ where $e\in G$ is the unit of $G$ by the formula
$(\varphi(x),h(x)\psi(x), h(x)\cdot e)$. Then $\mu:X\to W$ is equivariant on $X_0$ and $C=G\mu(X)$ is a camomile.

In fact the same construction of the camomile is valid for finite-dimensional separable metric space $X$ and closed subspace $X_0$ (with $G$-action).

Camomile is convenient for proving results of Bourgin--Yang type.

\begin{theorem} Assume that $Y$ is a $G$-space, $Y_0$ its invariant closed subspace such that the action on $Y\setminus Y_0$ is free, and $f:X\to Y$ a
continuous map. If

1) $n=\ind X_0> \ind (Y\setminus Y_0)$,

2) $X_0$ is an $n$-c.t.-subspace of $X$ over $\mathbb Z$,

3) $f|_{X_0}:X_0\to Y$ is equivariant,

\noindent
then $\dim f^{-1}(Y_0)\ge n-\ind (Y\setminus Y_0)$.
\end{theorem}
\begin{proof} We have $\ind C=n+1$, where $C$ is the camomile. Denote by $h:C\to Y$ the equivariant extension of $f$. Then $\ind h^{-1}Y_0\ge \ind C -\ind (Y\setminus P)-1=n-\ind (Y\setminus Y_0)$, hence
$\dim h^{-1}Y_0\ge n-\ind (Y\setminus Y_0)$. Since
$h^{-1}Y_0=\bigcup_{g\in G} g\cdot f^{-1}(Y_0)$ and $\dim g\cdot f^{-1}(Y_0)=\dim f^{-1}(Y_0)$ for any $g\in G$, we are done.
\end{proof}

Since a free $G$-space is a free space with respect to any subgroup we have
analogs of the above results in which $\ind (\,\cdot\,)$ is replaced by $\indp (\,\cdot\,)$ where $p=|H|$ is a prime and $H$ is some subgroup of $G$. For example we have the following result:
\begin{theorem}\label{ZYPp} Let $Y$ be a $G$-space, $Y_0$ its invariant closed subspace such that the action on $Y\setminus Y_0$ is free, and $f:X\to Y$ a continuous map. Let $H=\mathbb Z_p$, $p$ is a prime, be a subgroup of $G$.
Assume that

1) $n=\indp X_0\ge \indp (Y\setminus Y_0)$,

2) $X_0$ is an $n$-c.t.-subspace of $X$ over $\mathbb Z_p$,

3) $f|_{X_0}:X\to Y$ is equivariant.

Then $f^{-1}(Y_0)\not=\emptyset$.

If $n=\indp X> \indp (Y\setminus Y_0)$ then $\dim f^{-1}(Y_0)\ge n-\indp (Y\setminus Y_0)$.
\end{theorem}

\section{Tucker type lemmas}

\subsection{Tucker type lemmas for $G$-spaces}

Let $X$ be a simplicial complex and $C$ be a finite set.
Recall that a $C$-labeling (coloring) of $X$ is a map $V(X)\to C$ of the vertex set $V(X)$ to $C$. For $C=G\times\{1,\dots,n\}$ we say that
we have a $(G,n)$-labeling.
Thus a $(G,n)$-labeling prescribes to each vertex some pair $(g,k)$ where $g\in G$ and $k\in \{1,\dots,n\}$.

Now we define equivariant labelings.

\begin{df} Let $X$ be a simplicial complex with a simplicial $G$-action, where $G$ is a finite group, and $C$ is a finite $G$-set.
An equivariant $C$-labeling (coloring) of $X$ is an equivariant map $V(X)\to C$ of the vertex set $V(X)$ to $C$. For $C=G\times\{1,\dots,n\}$, where
$G$ acts on the first factor by left multiplication and on the second factor the action is trivial, we call $C$-labeling as equivariant $(G,n)$-labeling.
\end{df}

\begin{df} An edge in $X$ is called complementary if labels
of its vertices belong to the same orbit in $C$. For $(G,n)$-labeling it means that vertices of a complementary edge have the form $(g_1,k)$ and $(g_2,k)$,
$g_1\not=g_2$, for some $k\in \{1,\dots,n\}$.
\end{df}

If $G=\mathbb Z_2\cong C_2=\{1,-1\}$ is the cyclic group of order $2$ then a $(G,n)$-labeling is just a Tucker labeling since there is an obvious bijection $(\pm 1,k)\leftrightarrow \pm k$ between sets $\{1,-1\}\times \{1,\ldots,n\}$ and $\{+1,-1,+2,-2,\ldots, +n,-n\}$. Under this identification an equivariant $(\mathbb Z_2,n)$-labeling becomes an equivariant $\{\pm 1,\dots,\pm n\}$-labeling and
a complementary edge (for $(\mathbb Z_2,n)$-labeled complex $X$) is just a complementary edge in Tucker's sense (the sum of its labels equals zero).




\begin{theorem}\label{G-Takker}
$\tind X\geq d$\, if and only if for any equivariant $(G,d)$-labeling of the vertex set of
an arbitrary equivariant triangulation of $X$ there exists a complementary edge.
\end{theorem}

\begin{proof}  Assume there is a $(G,d)$-labeling of the vertex set of
an equivariant triangulation of $X$ without complementary edges. Then such a labeling
provides an equivariant map $X\to J^d(G)$. This contradicts with the assumption that $\tind X\geq d$.

Now assume that for any equivariant $(G,d)$-labeling of the vertex set of
an arbitrary equivariant triangulation of $X$ there exists a complementary edge.
Assume that $\tind X< d$. Then there exists an equivariant continuous map $X\to J^d(G)$ and
an equivariant simplicial approximation of this map which is a simplicial map of some triangulation of $X$.
So there exists a $(G,d)$-labeling of $J^d(G)$ without complementary edges. Thus, the inverse image of this labeling
is a $(G,d)$-labeling of $X$ without complementary edges, a contradiction.
\end{proof}
\begin{remark} {\rm
1)\, From equivariance of labeling it follows that if there exists a complementary edge, then there is a whole orbit of complementary edges, i.e.
there exist at least $|G|$ different complementary edges.

2)\, The first statement of theorem~\ref{G-Takker} holds for any numerical index $\Ind(\,\cdot\,)$ which possesses the main property of index and dimension property, i.e.:

A) If there exists an equivariant map $X\to Y$ of $G$-spaces then $\Ind(X)\le \Ind(Y)$.

B) $\Ind J^{d+1}(G)=d$.

All indexes considered above satisfy these properties.

Thus the following assertion holds:
}

{\it Let $X$ be a simplicial complex with a free simplicial $G$-action such that\, $\Ind(X)\le d$. Then for any equivariant $(G,d)$-labeling of the vertex set of
$X$ there exists a complementary edge (actually there exist at least $|G|$ complementary edges).}

\end{remark}

\subsection{Tucker type lemmas for bounded spaces}

Consider the case of simplicial complex $X$ and its subcomplex $X_0$. We assume that $G$ acts freely and simplicially on $X_0$.

\begin{theorem} Assume that $\ind X=n-1$ and that $X_0$ is an $(n-1)$-c.t.-subspace of $X$ over $\mathbb Z$.
Then for any $(G,n)$-labeling of the vertex set of an arbitrary triangulation of $X$
which is equivariant on $X_0$ there exists a complementary edge.
\end{theorem}
\begin{proof} We argue by contradiction. A $(G,n)$-labeling of the vertex set of a triangulation of $X$ without complementary edges provides
a map $\psi:X\to J^n(G)$, and this map is equivariant on $X$ since our $(G,n)$-labeling is equivariant on $X$. Since $i^*$ is trivial in dimension $n-1$, where $i:X_0\subset X$ is the inclusion, we see that
$(\psi|_X)^*:H^{n-1}(J^n(G);\mathbb Z)\to H^{n-1}(X;\mathbb Z)$ is trivial, and we obtain a contradiction with property~8 of cohomological index,
because $\ind X=n-1=\ind J^n(G)$.
\end{proof}

As a partial case we obtain:
\begin{theorem} Let $M^n$ be a connected compact orientable $PL$--manifold such that its boundary $\partial M$ is homeomorphic to the sphere ${\Bbb S}^{n-1}$. Let $T$ be a triangulation of $M$. Suppose that  there exists a free simplicial action of a finite group $G$ on $\partial T$. Then for any $(G,n)$--labeling of $V(T)$ that is an equivariant on  $\partial T$ there exists a complementary edge.
\end{theorem}

\subsection{Bourgin--Yang type results and Tucker type lemmas}

Consider first the case $G=\mathbb Z_2$.
By $\inn (\,\cdot\,)$ and $\tind (\,\cdot\,)$, $G=\mathbb Z_2$, we denote Yang's homological and topological indexes respectively, so
$\inn ({\Bbb S}^n)=\tind ({\Bbb S}^n)=n$.

\begin{prop} Let $f: {\Bbb S}^n\to M$ be a continuous map to $m$-dimensional manifold $M$ where $d:=n-m>0$. Assume that $K:=\{x\in {\Bbb S}^n\,\,|\,\, f(x)=f(-x)\}$ is a triangulable space. Then for any equivariant triangulation of $K$ and an equivariant labeling of its vertices by $\{+1,-1,+2,-2,\ldots, +d,-d\}$ there exists a complementary edge.
\end{prop}
\begin{proof}
It follows from \cite{Nakaoka} (see also \cite{Vo79}) that $\inn K\ge d$, and by Theorem~\ref{G-Takker} we are done.
\end{proof}
Note that if $M$ is a $PL$-manifold and $f$ is simplicial with respect to some equivariant triangulation of ${\Bbb S}^n$, then $K$ is a triangulable space (but not necessarily a simplicial subspace).

\medskip

Similar result holds for $G=\mathbb Z_p\oplus\dots\oplus\mathbb Z_p$, elementary Abelian group, $p$ is a prime.
\begin{prop} Let $f: J^{n+1}(G)\to M$ be a continuous map to $m$-dimensional manifold $M$ where $d:=n-m(|G|-1)>0$ and $G=\mathbb Z_p\oplus\dots\oplus\mathbb Z_p$ is an elementary Abelian group, $p$ is a prime. Assume that  $K:=\{x\in J^{n+1}(G)\,\,|\,\, f(x)=f(gx)\,\,\forall\,\,g\in G\}$ is a triangulable space.
Then for any equivariant triangulation of $K$ and an equivariant $(G,d)$-labeling of its vertices there exists a complementary edge.
\end{prop}
\begin{proof}
It follows from~\cite{Vo05} that $\tind K\ge n-m(|G|-1)$. Hence the result follows from Theorem~\ref{G-Takker}.
\end{proof}

 \medskip

O.\,R. Musin,  University of Texas Rio Grande Valley,  Brownsville, TX, 78520\\
{\it E-mail address:} oleg.musin@utrgv.edu

 \medskip

A.\,Yu. Volovikov,  Department of Higher Mathematics, MTU MIREA, Vernadskogo av., 78, Moscow, 119454, Russia\\
{\it E-mail address:} $\mbox{a\_volov@list.ru}$

\end{document}